\documentclass[11pt,reqno]{amsart}
\usepackage[T1]{fontenc}
\usepackage{lmodern}
\usepackage[expansion=false]{microtype}
\usepackage{mathtools}
\usepackage{amssymb}
\usepackage{booktabs}
\usepackage{url}
\usepackage[hidelinks]{hyperref}

\pdftrailerid{}
\hypersetup{pdftitle={A construction of F-irregular graphs},pdfauthor={James Alexander Schreib}}
\numberwithin{equation}{section}
\newtheorem{theorem}{Theorem}[section]
\newtheorem{lemma}[theorem]{Lemma}
\newtheorem{corollary}[theorem]{Corollary}

\newcommand{\doi}[1]{\url{https://doi.org/#1}}

\title{A construction of $F$-irregular graphs}
\author{James Alexander Schreib}
\address{Department of Mathematics, New York University, New York, NY 10012}
\email{jas10320@nyu.edu}
\thanks{ORCID: \url{https://orcid.org/0009-0005-1048-0209}.}
\date{5 September 2026}

\begin{document}
\begin{abstract}
For a fixed graph $F$, the $F$-degree of a vertex $v$ in a host graph $H$ is the number of subgraphs of $H$ isomorphic to $F$ that contain $v$. A host is $F$-irregular if its $F$-degrees are pairwise distinct. We show that every finite connected graph $F$ on at least three vertices admits a finite connected $F$-irregular host with at least two vertices. For noncomplete $F$, the proof builds the host as a threshold graph with a deleted edge, augmented by a small incidence gadget whose weighted column sums separate the finitely many vertices the threshold part cannot. The complete case uses the existence theorem of Chartrand, Holbert, Oellermann and Swart (1987). A Lean~4 formalization, taking this published theorem as its sole custom axiom, is described at the end of the paper.
\end{abstract}
\maketitle

\section{Introduction}
\label{sec:intro}

No nontrivial graph has all vertex degrees distinct \cite{behzad1967}.
The search for usable substitutes produced several notions of irregularity:
Chartrand, Erd\H{o}s and Oellermann compared candidate definitions \cite{chartrand1988},
and Alavi, Chartrand, Chung, Erd\H{o}s, Graham and Oellermann introduced highly
irregular graphs, which are connected graphs in which every vertex is adjacent only to vertices of distinct
degree \cite{alavi1987}.
Here we work with the relative notion.
Fix a graph $F$, following Chartrand, Holbert, Oellermann and Swart~\cite{chartrand1987}. For a host graph $H$ and $v \in V(H)$, the \emph{$F$-degree}
$d_F^H(v)$ is the number of subgraphs of $H$ isomorphic to $F$ with $v$ in their
vertex set. The host $H$ is \emph{$F$-irregular} if the values $d_F^H(v)$ are
pairwise distinct. That every finite connected $F$ on at least three vertices admits such a host with at least two vertices was conjectured by Chartrand, Holbert, Oellermann and Swart in 1987~\cite{chartrand1987}.

The complete and star patterns were settled in the original paper~\cite{chartrand1987}.
More recently, Dovzhenok, Filuta and Chuhai~\cite{dovzhenok2024} proved
existence of infinitely many hosts for biconnected patterns of minimum degree
two. Dovzhenok proved the corresponding infinite-family result for patterns
of diameter two~\cite{dovzhenokdiameter2026}, and settled existence for
paths~\cite{dovzhenokpaths2026}, subsequently obtaining infinite families
with simultaneous ordinary and rooted path irregularity~\cite{dovzhenokrooted2026}.
Theorem~\ref{thm:main} establishes existence for all connected patterns of
order at least three, with a connected host.

\begin{theorem}
\label{thm:main}
For every finite connected graph $F$ with at least three vertices there exists
a finite connected $F$-irregular host graph $H$ with at least two vertices.
\end{theorem}

The complete pattern is settled by the existence theorem of
Chartrand, Holbert, Oellermann and Swart~\cite{chartrand1987}, together with
the component reduction in Section~\ref{sec:prelim}; Sections~\ref{sec:matrix}--\ref{sec:cap} treat the
noncomplete case. The device is a rooted count $C_H(v)$, proportional to
$d_F^H(v)$ with a constant factor, so distinguishing $C_H$ suffices.
Most vertices are ordered by neighborhood inclusion in a threshold host with
one deleted edge (Section~\ref{sec:chain}). Two vertices displaced by that deletion
are reinserted by a local automorphism comparison (Section~\ref{sec:displaced}), and
the finitely many remaining vertices are separated by an incidence gadget with distinct weighted column sums. All error
terms are absorbed by taking the host large, as fixed in
Section~\ref{sec:size}.

\section{Preliminaries}
\label{sec:prelim}

All graphs are finite, simple and undirected; background terminology follows
Diestel \cite{diestel2017}. A graph is nontrivial if it has at least two vertices.
Subgraphs are not required to be induced; different edge sets on the same
vertex set can define different copies.

For every $k\ge3$, a finite nontrivial $K_k$-irregular graph exists
\cite{chartrand1987}; this result is also explicitly stated
in~\cite[p.~42]{chartrand1988}. To obtain a connected host, choose a vertex
of positive $K_k$-degree, which exists because a nontrivial irregular host
has at most one vertex of degree zero in this count. Its component contains
a copy of $K_k$ and is therefore nontrivial. Every copy of the connected
pattern containing a vertex of that component lies wholly in it, so
restriction preserves all its $K_k$-degrees and hence irregularity.

For the construction below, assume $F$ is noncomplete and connected,
write $k = \lvert V(F) \rvert$
and $\delta = \delta(F)$, and assume $k \ge 3$. A
shortest path between nonadjacent vertices furnishes an induced $P_3$, and the
endpoints of any nonedge have degree at most $k-2$, so $\delta \le k-2$.

For a host $H$, let $C_H(v)$ count injective homomorphisms $\varphi \colon F \to H$
whose image contains $v$. Each subgraph counted by $d_F^H(v)$ lifts to exactly
$\lvert \operatorname{Aut} F \rvert$ such maps, so
$C_H(v) = \lvert \operatorname{Aut} F \rvert \, d_F^H(v)$, and it suffices to
make the values $C_H(v)$ pairwise distinct.

Write $(x)_j = x(x-1)\dotsm(x-j+1)$ with $(x)_0 = 1$. The following estimate is
used throughout: if $U \subset V(H)$ has $M$ vertices and $\lvert V(H) \rvert = N$,
then the number of injective maps $V(F) \to V(H)$ meeting $U$ in at least $s$
vertices is at most
\begin{equation}
\label{eq:countbd}
\binom{k}{s} M^s N^{k-s}.
\end{equation}
Indeed, choose $s$ preimages, their images in $U$, and arbitrary images for the
rest; overcounting is harmless. Every appeal to \eqref{eq:countbd} below takes
$U$ to be a fixed finite exceptional set while $N \le 2n$ with $n$ as in
\eqref{eq:size}.
\section{A finite incidence matrix}
\label{sec:matrix}

Assume $\delta \ge 2$. Put $c = 2\delta+1$ and $r = \delta+1$, and construct
subsets $S_t \subset \{1,\dots,c\}$, $1 \le t \le r$, of sizes
$d_t = \delta+t-1$ as follows: start with $S_t = \{1,\dots,d_t\}$, then for each
$j = 1,\dots,\delta-1$ replace $j$ by $c$ in $S_{j+1}$.
Hence column $j < \delta$ is absent from exactly $S_{j+1}$, column $\delta$ lies in every $S_t$, column $\delta+s$ is absent from exactly $S_1,\dots,S_s$, and $c$ lies in exactly $S_2,\dots,S_\delta$.

Put $w_t = \binom{d_t-1}{\delta-1}$ and define the weighted column sums
$W_j = \sum_{t : j \in S_t} w_t$.

\begin{lemma}
\label{lem:columns}
The $c$ sums $W_j$ are pairwise distinct.
\end{lemma}

\begin{proof}
Write $T = \sum_t w_t$, $P_s = \sum_{t \le s} w_t$ and
$U = \sum_{t=2}^{\delta} w_t$. A short inspection of the construction gives
\[
T - w_{j+1}\quad (1 \le j < \delta), \qquad T\quad (j = \delta),
\]
\[
T - P_s\quad (j = \delta+s,\ 1 \le s \le \delta), \qquad U\quad (j = c).
\]
For $2 \le t \le \delta$, the hockey-stick identity \cite{stanley2012} gives
$P_{t-1} = \frac{t-1}{\delta} w_t$, whence $P_{t-1} < w_t < P_t$; no $w_t$ from
the first group therefore equals any $P_s$. All sums but $U$ are at least $w_r$,
whereas $P_{\delta} = w_r$ and $w_1 = 1$ give $U = w_r - 1$. Distinctness within
each group is immediate from strict monotonicity of the $w_t$.
\end{proof}

\section{Construction and choice of size}
\label{sec:size}

For $\delta \ge 2$ use $c$, $r$ and the sets $S_t$ above; for $\delta = 1$ put
$c = 1$ and $r = 0$, so there are no further rows. Define
$C_3 = \binom{k}{3} 8^3$, $C_4 = \binom{k}{4} 8^4$ with $\binom{k}{4} = 0$ for
$k = 3$, and $C^* = \binom{k}{\delta+2}(3\delta+2)^{\delta+2}$ for
$\delta \ge 2$, $C^* = 0$ for $\delta = 1$. Choose
\begin{equation}
\label{eq:size}
n = 2(c+k+4) + 2^{2k}(1+C_3+C_4+C^*).
\end{equation}
Take a clique on $B_1,\dots,B_n$ and an independent set
$A_1,\dots,A_m$ with $m = n-c+1$ and $N(A_i) = \{B_i,\dots,B_n\}$.
Write $Q = \{B_{n-c+1},\dots,B_n\}$ as $\{q_1,\dots,q_c\}$.
For $\delta \ge 2$ add independent vertices $T_1,\dots,T_r$, joining $T_t$
to those $q_j$ with $j \in S_t$; put no edges among the $A_i$ and $T_t$.
Delete the edge $B_1B_4$. Call the result $G$ and put
$N = \lvert V(G) \rvert \le 2n$.

By \eqref{eq:size}, the falling factorials $(n-4)_{k-2}$,
$(n-4)_{k-3}$ and, when $\delta\ge2$, $(n-c-4)_{k-\delta-1}$
are at least $(n/2)$ to their respective degrees, with strict inequality
for positive degree. Since $n>2^{2k}(1+C_3+C_4+C^*)$,
these bounds yield
\begin{equation}
\label{eq:gap3}
2\lvert E(F) \rvert (n-4)_{k-2} > 2C_3(2n)^{k-3},
\end{equation}
\begin{equation}
\label{eq:gap4}
(n-4)_{k-3} > 2C_4(2n)^{k-4} \qquad (k \ge 4),
\end{equation}
and, for $\delta \ge 2$ with $\ell = k-\delta-1$,
\begin{equation}
\label{eq:gapstar}
\delta!\,(n-c-4)_{\ell} > 2C^*(2n)^{\ell-1}.
\end{equation}
For $k = 3$ the error in \eqref{eq:gap4} is zero and the corresponding surplus
is $1$.
\section{Comparisons from neighborhood inclusion}
\label{sec:chain}

If $N(u) \setminus \{v\} \subset N(v) \setminus \{u\}$, replacing $u$ by $v$
injects the embeddings rooted at $u$ but not using $v$ into those rooted at $v$
but not using $u$. Embeddings containing both vertices contribute equally
to the two rooted counts and cancel in their difference.
A valid embedding through $v$ but not $u$ that uses an edge unavailable
at $u$, and lies outside the image of this injection, makes the inequality
strict. Applied to $G$ this yields
\begin{equation}
\label{eq:chain}
\begin{split}
\text{every } T_t < A_m < \dots < A_3 < A_2 < A_1 \\
< B_2 < B_3 < B_5 < \dots < B_{n-c} < \text{every } q_j,
\end{split}
\end{equation}
where $<$ compares $C_G$ values and the $T$ part is absent for $r = 0$.
Here are strict witnesses. Fix a vertex $a$ of $F$ of degree $\delta$.
For $A_{i+1}<A_i$ ($1\le i<m$), send $a$ to $A_i$, one neighbour to $B_i$,
and the other $\delta-1$ neighbours into $Q$. The edge $A_iB_i$ is
unavailable at $A_{i+1}$.

For consecutive listed vertices $B_i<B_j$, send $a$ to $A_j$, one
neighbour to $B_j$, and the other $\delta-1$ neighbours into $Q$.
Then $A_j$ is adjacent to $B_j$ but not $B_i$.
For $A_1<B_2$, use the same assignment with $a$ at $A_2$ and one
neighbour at $B_2$; the edge $A_2B_2$ is unavailable at $A_1$.

For $B_{n-c}<q_j$, send $a$ to $A_m$, one neighbour to $q_j$,
and the other $\delta-1$ neighbours into $Q\setminus\{q_j\}$.
The edge $A_mq_j$ is unavailable at $B_{n-c}$.
For $T_t<A_m$, choose $q_j\in Q\setminus N(T_t)$, which exists since
$d_t\le2\delta<c$, and use the same assignment. The edge $A_mq_j$
is unavailable at $T_t$.

Map the remaining vertices injectively into the clique
$B_5,\dots,B_{n-c-1}$, avoiding already used vertices and the smaller
comparison vertex. There are enough vertices by \eqref{eq:size}, and
$c\ge\delta$ supplies the required neighbours in $Q$. These $B$ images
and the images of the neighbours of $a$ form a clique avoiding the deleted
edge. Since all neighbours of $a$ have been assigned, every required edge
is preserved. Each witness avoids the smaller comparison vertex and
cannot be obtained by replacing it with the larger one.

\section{The two displaced vertices}
\label{sec:displaced}

Put $S = \{A_1,A_2,A_3,A_4,B_1,B_2,B_3,B_4\}$. Every vertex of $S$ has the same
neighbourhood outside $S$, namely the clique $B_5,\dots,B_n$. For $u,v \in S$,
embeddings meeting $S$ exactly once contribute equally to $C_G(u)$ and $C_G(v)$.
Those meeting $S$ exactly twice contribute a difference of
$(\deg_{G[S]}(u) - \deg_{G[S]}(v)) L_n$ with
$L_n \ge 2\lvert E(F) \rvert (n-4)_{k-2}$: for fixed second vertex of $S$ the
count depends only on adjacency of the two distinguished vertices, and the
adjacent case gains at least the maps sending an ordered edge of $F$ to them
with everything else in $B_5,\dots,B_n$. By \eqref{eq:countbd}, embeddings
meeting $S$ at least three times differ by at most $2C_3(2n)^{k-3}$ in absolute
value. Since $\deg_S(B_1) - \deg_S(A_3) = \deg_S(B_4) - \deg_S(B_2) = 1$,
\eqref{eq:gap3} gives
\begin{equation}
\label{eq:degdiff}
C_G(A_3) < C_G(B_1) \quad\text{and}\quad C_G(B_2) < C_G(B_4).
\end{equation}

To place $B_1$ against $A_2$, set $x = B_1$, $y = A_2$, $a = A_1$, $b = B_4$ and
temporarily delete $bA_3$ and $bA_4$. The pruned graph admits $(x\,y)(a\,b)$ as
an automorphism, so the two rooted counts agree there. Restoring the two edges,
any newly valid embedding through either root uses an added edge. With exactly
three vertices of $S$ these are the root, $b$ and $z \in \{A_3,A_4\}$; replacing
$x$ by $y$ preserves every required edge, because the external neighbourhoods
agree, neither root touches $z$, and $yb$ is present while $xb$ is absent. The
new $x$-embeddings thus inject into the new $y$-embeddings. An induced $P_3$ of
$F$ sent successively to $y,b,z$ with the rest in $B_5,\dots,B_n$ furnishes at
least $(n-4)_{k-3}$ embeddings outside that injection. Embeddings meeting $S$
at least four times differ by at most $2C_4(2n)^{k-4}$, and by $0$ for $k = 3$.
With \eqref{eq:gap4},
\begin{equation}
\label{eq:b1a2}
C_G(B_1) < C_G(A_2).
\end{equation}
For $B_4$ against $B_3$ put instead $x = B_4$, $y = B_3$, $a = A_4$, $b = B_1$
and delete $bA_1$, $bB_2$; again $(x\,y)(a\,b)$ is an automorphism. The same
argument applies with $z \in \{A_1,B_2\}$, both roots now adjacent to $z$ while
$yb$ is present and $xb$ absent, and the same induced $P_3$ witnesses strictness
since extra host edges are permitted. Hence \eqref{eq:gap4} gives
\begin{equation}
\label{eq:b4b3}
C_G(B_4) < C_G(B_3).
\end{equation}
Together \eqref{eq:degdiff}, \eqref{eq:b1a2} and \eqref{eq:b4b3} refine \eqref{eq:chain} to
\begin{equation}
\label{eq:chainend}
A_4 < A_3 < B_1 < A_2 < A_1 < B_2 < B_4 < B_3.
\end{equation}
\section{Tails and the cap}
\label{sec:cap}

Only $\delta \ge 2$ needs this step. Put $D = Q \cup \{T_1,\dots,T_r\}$,
$R = V(G) \setminus D$ and $\ell = k-\delta-1 \ge 1$. The set $Q$ is a clique
complete to $R$: each cap vertex touches every $B$ outside $Q$ and every $A_i$.
Define
\[
K_n = \delta! \sum_{a : \deg_F(a) = \delta} \operatorname{inj}(F - N_F[a], G[R]),
\]
with $\operatorname{inj}$ counting injective homomorphisms and $N_F[a]$ the
closed neighbourhood. Since $B_5,\dots,B_{n-c}$ is a clique in $R$,
\begin{equation}
\label{eq:kn}
K_n \ge \delta!\,(n-c-4)_{\ell}.
\end{equation}
The rooted counts decompose as
\begin{equation}
\label{eq:tail}
C_G(T_t) = K_n \binom{d_t}{\delta} + E_t,
\end{equation}
\begin{equation}
\label{eq:cap}
C_G(q_j) = H_n + K_n W_j + E_j^{\prime},
\end{equation}
where $H_n$ is independent of $j$ and every error satisfies
\begin{equation}
\label{eq:err}
0 \le E_t, E_j^{\prime} \le C^*(2n)^{\ell-1}.
\end{equation}
Indeed, a vertex of $F$ sent to a tail forces all its neighbours into $Q$, so
any embedding meeting a tail uses at least $\delta+1$ vertices of $D$. With
exactly $\delta+1$, there is one tail image whose preimage $a$ has degree
$\delta$, its $\delta$ neighbours land exactly on its $Q$ neighbours, and the
rest lands in $R$. For \eqref{eq:tail} the $\delta$ neighbours can be assigned
in $(d_t)_{\delta}$ ways; for \eqref{eq:cap} the cap root is one of those
neighbours, giving $\delta(d_t-1)_{\delta-1}$ assignments when $q_j$ touches
$T_t$ and $0$ otherwise. These are exactly the displayed $K_n$ coefficients.
Any further embedding meeting a tail uses at least $\delta+2$ vertices of $D$;
in particular two tail images force $\delta$ further distinct $Q$ images.
With $\lvert D \rvert = 3\delta+2$, \eqref{eq:countbd} gives \eqref{eq:err}.
Embeddings avoiding all tails contribute the common $H_n$ in \eqref{eq:cap},
since the cap vertices are interchangeable once the tails are removed.

The integers $\binom{d_t}{\delta}$ grow strictly in $t$, and the $W_j$ are
distinct by Lemma~\ref{lem:columns}. Hence \eqref{eq:kn}--\eqref{eq:err} with
\eqref{eq:gapstar} separate every pair of tails and every pair of cap vertices.
With \eqref{eq:chain} and \eqref{eq:chainend} this orders every vertex of $G$
by $C_G$, so the $F$-degrees are pairwise distinct. The graph $G$ is finite and
connected and has at least two vertices, proving the noncomplete case;
with the complete case established in Section~\ref{sec:prelim} this proves
Theorem~\ref{thm:main}.

\begin{corollary}
\label{cor:infinite}
For every finite connected noncomplete graph $F$ with at least three vertices,
there exist infinitely many pairwise non-isomorphic finite connected
$F$-irregular graphs.
\end{corollary}

\begin{proof}
Keep $F$, $c$ and $r$ fixed and let $n$ range over all integers at least
the value in \eqref{eq:size}. The same bounds and strict comparisons apply
for every such $n$, so each resulting host $G_n$ is connected and
$F$-irregular. Since $\lvert V(G_n)\rvert=2n-c+1+r$, distinct values of $n$
give non-isomorphic hosts.
\end{proof}

\section{Tool and computational resource disclosure}

ChatGPT~6 Astra (OpenAI) was accessed in chat to generate the proof.
ChatGPT~5.6 Sol was accessed through codex to write portions of the Lean~4
development described in Section~\ref{sec:formal}\footnote{On the emerging
conventions for such disclosure, see the Leiden Declaration on Artificial
Intelligence and Mathematics~\cite{leiden2026}}.
All statements and proofs were subsequently verified independently by the
author. The formal development is checked by the Lean kernel, and the Lean
statements have been confirmed by hand to faithfully restate the corresponding
statements of this paper. The author alone is responsible for the correctness
of the arguments and for the accuracy and completeness of the citations.

\section{Formal verification}
\label{sec:formal}

The proof of Theorem~\ref{thm:main} is formalized in Lean~4~\cite{demoura2021} using
mathlib~\cite{mathlib2020}, with one custom axiom:
\path{complete_pattern_irregular_host}. This records the published
complete-pattern existence theorem of Chartrand, Holbert, Oellermann and
Swart~\cite{chartrand1987}, also stated in~\cite[p.~42]{chartrand1988}:
it supplies a finite nontrivial $F$-irregular host
whenever $F$ is complete and has at least three vertices. Lean proves that
such a host can be restricted to a connected component. The public endpoint
is \path{FIrregular.f_irregular_conjecture}.

Everything else in the proof of Theorem~\ref{thm:main} is proved in Lean. This includes the
non-induced rooted-copy semantics and automorphism quotient, the explicit
host and incidence matrix, the minimum-degree-one and higher-degree cases,
the count decompositions and error bounds, every strict comparison, and the
final connected host witness for noncomplete $F$. Thus, relative to the one
axiom above, the formalization is sufficient to derive the theorem stated in
this paper. The varying-$n$ consequence in Corollary~\ref{cor:infinite}
is not included in the current formalization.

The guarded audit in \texttt{AxiomAudit.lean} verifies that the noncomplete
endpoint uses no custom axioms and that the public theorem uses exactly
\path{complete_pattern_irregular_host}, in addition to Lean's standard
foundations. The sources contain no \texttt{sorry}, \texttt{admit}, or
\texttt{native\_decide}. Running \texttt{lake build} under Lean~4.30.0 and
mathlib~v4.30.0 reproduces the check.
\url{https://github.com/jamesschreib/f-irregular-graph-conjecture}
is the public development repository~\cite{schreib2026}.

\end{document}